\documentclass[a4paper,reqno,12pt]{amsart}
\usepackage{verbatim}
\usepackage{amssymb,amsmath,amsthm}
\usepackage{amsfonts}
\usepackage{array}

\newtheorem{theorem}{Theorem}
\newtheorem{lemma}{Lemma}

\theoremstyle{remark}
\newtheorem{remark}[theorem]{\bf Remark}

\newcommand*\pFq[5]{{}_{#1}F_{#2}{\left[\genfrac..{0pt}{}{#3}{#4};#5\right]}}

\def\and{\quad\mbox{and}\quad}
\begin{document}

\setcounter{page}{1}

\title[]{Supercongruences related to ${}_3F_2(1)$ \\
involving harmonic numbers   
}

\author{Roberto~Tauraso}

\address{Dipartimento di Matematica, % \\
Universit\`a di Roma ``Tor Vergata'', % \\
via della Ricerca Scientifica,%\\
00133 Roma, Italy}
\email{tauraso@mat.uniroma2.it}



\date{\today}

\begin{abstract} We show various supercongruences for truncated series which involve central binomial coefficients and harmonic numbers. The corresponding infinite series are also evaluated.
\end{abstract}

\maketitle

\section{Introduction}
In 1997, Van Hamme \cite{VH97} established the $p$-adic analogs of several Ramanujan type series. For one of them, the series labeled (H.1),  
\begin{equation}\label{E20}
\sum_{k=0}^{\infty}\frac{\binom{2k}{k}^3}{64^k}=
\pFq{3}{2}{\frac{1}{2},\frac{1}{2},\frac{1}{2}}{1,1}{1}=\frac{\pi}{\Gamma^4(\frac{3}{4})},
\end{equation}
the modulo $p^2$ congruence (H.2) for the truncated version has been recently improved by Long and Ramakrishna in \cite[Theorem 3]{LR16},
\begin{equation}\label{E21}
\sum_{k=0}^{p-1}\frac{\binom{2k}{k}^3}{64^k}\equiv_{p^3} \left \{ \begin{array}{ll} - \Gamma^4_p\left(\frac{1}{4}\right) & \text{ if } p\equiv_4 1, \vspace{3mm}\\
 -\frac{p^2}{16}\Gamma_p^4\left(\frac 14\right)  
& \text{ if } p\equiv_4 3. \end{array} \right.
\end{equation}
where $p$ is any prime greater than $3$
(we use the notation $a\equiv_m b$ to mean $a\equiv b \pmod{m}$).

In this paper we will investigate the series and the corresponding partial sums where the terms have one of the following forms
$$\binom{2k}{k}^3\frac{H_k}{64^{k}},\quad
\binom{2k}{k}^3\frac{H_k^{(2)}}{64^{k}},\quad
\binom{2k}{k}^3\frac{O_k}{64^{k}},\quad
\binom{2k}{k}^3\frac{O_k^{(2)}}{64^{k}}.$$
Here $H_k^{(r)}$ denotes the $k$-th generalized harmonic number of order 
$r$ and  $O_k^{(r)}$ is the sum with odd denominators,
$$H_k^{(r)}=\sum_{j=1}^{k}\frac{1}{j^r}
\quad\mbox{and}\quad O_k^{(r)}=\sum_{j=1}^{k}\frac{1}{(2j-1)^r}$$
where we adopt the convention that $H_k=H_k^{(1)}$ and $O_k=O_k^{(1)}$.

The main results are presented in Section 3 (evaluations of the infinite series) and Section 5 (congruences for the truncated series). For example we show that 
$$\sum_{k=1}^{\infty}\binom{2k}{k}^3\frac{O_k}{64^k}=\frac{ \pi^2}{6\Gamma^4(\frac{3}{4})},\and
\sum_{k=1}^{p-1}\binom{2k}{k}^3\frac{O_k}{64^k}
\equiv_{p^2}
\begin{cases}
0
&\text{if $p\equiv_4 1$},\vspace{3mm}\\
-\frac{p}{12}\Gamma^4_p\left(\frac{1}{4}\right) 
&\text{if $p\equiv_4 3$}.
\end{cases}$$
The correspondence between the right-hand sides of the infinite series and the finite sum is particularly striking for the appearance of the classic Gamma function and the $p$-adic analog.

\section{Similar results of lower degree}

Before dealing with the main issue, we are going to take a look to similar sums already in the literature, where the central binomial coefficient is raised to a power less than $3$. Assume that $p$ is a prime greater than $3$.
For $n\geq 1$, we have
$$\sum_{k=1}^{n-1}\binom{2k}{k}\frac{H_k}{4^{k}}
=\binom{2n}{n}\frac{2n(H_{n-1}-2)}{4^{n}}+2.
$$
Thus, by $n=p$, we obtain (see \cite[(1.10)]{Szw15} for the modulo $p^3$ version)
$$\sum_{k=1}^{p-1}\binom{2k}{k}\frac{H_k}{4^{k}}
\equiv_{p^4}2-2p+4p^2q_p(2)-6p^3q^2_p(2)-\frac{1}{3}p^3B_{p-3},
$$
where $q_p(a)=\frac{a^{p-1}-1}{p}$ is the Fermat quotient and we used  the Wolstenholme's theorem $\binom{2p}{p}\equiv_{p^3} 2$, and the congruences
\begin{equation}\label{E12}
H_{p-1}\equiv_{p^3}-\frac{1}{3}p^2B_{p-3},\quad
4^{p-1}\equiv_{p^3}1+2pq_p(2)+p^2q^2_p(2)
\end{equation}
(for the first one we can refer to \cite[Theorem 5.1 (a)]{Szh00}).
Moreover, the identity
$$\sum_{k=1}^{n-1}\binom{2k}{k}\frac{H_k^{(2)}}{4^{k}}
=\binom{2n}{n}\frac{2nH_{n-1}^{(2)}}{4^{n}}-2
\sum_{k=1}^{n-1}\frac{\binom{2k}{k}}{k4^{k}}
$$
implies (see \cite[(1.11)]{Szw15} for the modulo $p$ version)
$$\sum_{k=1}^{p-1}\binom{2k}{k}\frac{H_k^{(2)}}{4^{k}}
\equiv_{p^3}-4q_p(2)+2pq^2_p(2)-\frac{4}{3}p^2q^3_p(2)
-\frac{1}{2}p^2B_{p-3},
$$
where we employed the congruence established in 
\cite[Theorem 1.1]{Ta10},
$$\sum_{k=1}^{n-1}\frac{\binom{2k}{k}}{k4^{k}}\equiv_{p^3},
-H_{\frac{p-1}{2}}$$
and
\begin{equation}\label{E13}
H_{p-1}^{(2)}\equiv_{p^2}\frac{2}{3}pB_{p-3},\quad
H_{\frac{p-1}{2}}\equiv_{p^3} -2q_p(2) +pq^2_p(2)-\frac{2}{3}p^2q^3_p(2)-\frac{7}{12}p^2B_{p-3}
\end{equation}
given in \cite[Theorem 5.1 (a)]{Szh00})
and  \cite[Theorem 5.2 (c)]{Szh00} respectively.

As regards the squared case, the identities \cite[(2.4) and (2.8)]{Pr08}
\begin{align*}
&\sum_{k=1}^{n}\binom{n}{k}\binom{n+k}{k}(-1)^kH_k=2(-1)^n H_n,\\
&\sum_{k=1}^{n}\binom{n}{k}\binom{n+k}{k}(-1)^kH_k^{(2)}=2(-1)^{n+1}\sum_{k=1}^n \frac{(-1)^k}{k^2},
\end{align*} 
and the congruence for $0\leq k\leq n=(p-1)/2$ (note that $p$ divides $\binom{2k}{k}$ for $n<k<p$)
\begin{equation}\label{E01}
\binom{n}{k}\binom{n+k}{k}(-1)^k
=\binom{2k}{k}\frac{\prod_{j=1}^k((2j-1)^2-p^2)}{4^k(2k)!}
\equiv_{p^2}
\frac{\binom{2k}{k}^2}{16^k}
\end{equation}
imply  \cite[Theorem 4.1]{Szw15} (see also \cite[Theorems 1.1 and 1.2]{Szw14} for a more general $p^2$-congruence) 
\begin{align*}
&\sum_{k=1}^{p-1}\binom{2k}{k}^2\frac{H_k}{16^k}\equiv_{p^2}(-1)^{\frac{p+1}{2}} (4q_p(2) -2pq^2_p(2)),\\
&\sum_{k=1}^{p-1}\binom{2k}{k}^2\frac{H_k^{(2)}}{16^k}\equiv_{p^2}-8E_{p-3}+4E_{2p-4},
\end{align*}
where we also used 
\begin{equation}\label{E04}
H_{\frac{p-1}{2}}^{(2)}\equiv_{p^2}\frac{7}{3}pB_{p-3},\quad
H_{\lfloor \frac{p}{4}\rfloor}^{(2)}\equiv_{p^2}(-1)^{\frac{p-1}{2}}(8E_{p-3}-4E_{2p-4})+\frac{14}{3}pB_{p-3}
\end{equation}
given in \cite[Corollary 5.2]{Szh00}, \cite[Corollary 3.8]{Szh08}
and 
\begin{equation}\label{E03}
\sum_{k=1}^{n} \frac{(-1)^k}{k^2}=
\frac{1}{2}H_{\lfloor \frac{p}{4}\rfloor}^{(2)}
-H_{\frac{p-1}{2}}^{(2)}\equiv_{p^2}(-1)^{\frac{p-1}{2}}(8E_{p-3}-4E_{2p-4}).
\end{equation}

\section{Evaluations of the infinite series}

The generalized hypergeometric function is defined as
\begin{align*}
\pFq{r}{s}{a_1,a_2,\cdots,a_r}{b_1,b_2,\cdots,b_s}{z}=\sum_{k=0}^{\infty}\frac{(a_1)_k (a_2)_k\cdots (a_r)_k}{(b_1)_k (b_2)_k\cdots (b_s)_k}\cdot \frac{z^k}{k!}
\end{align*}
where
$(x)_k=x(x+1)\cdots (x+k-1)$ for $k\ge 1$ and $(x)_0=1$ is the  Pochhammer symbol and $a_i$, $b_j$ and $z$ are complex numbers with none of the $b_j$ being negative integers or zero. 
We recall some well-known hypergeometric identities:\vspace{1mm}

\noindent i) Dixon's theorem \cite[p.13]{Ba35}
\begin{equation}\label{Dixon}
\pFq{3}{2}{a,b,c}{1+a-b,1+a-c}{1}=\frac{\Gamma(1+\frac{a}{2})\Gamma(1+a-b)\Gamma(1+a-c)\Gamma(1+\frac{a}{2}-b-c)}{
\Gamma(1+a)\Gamma(1+\frac{a}{2}-b)\Gamma(1+\frac{a}{2}-c)\Gamma(1+a-b-c)},
\end{equation}

\noindent ii) Whipple's theorem \cite[p.16]{Ba35} 
\begin{equation}\label{Whipple}
\pFq{3}{2}{a,1-a,c}{e,1+2c-e}{1}=\frac{\pi2^{1-2c}\Gamma(e)\Gamma(1+2c-e)}{
\Gamma(\frac{a+e}{2})\Gamma(\frac{1-a+e}{2})\Gamma(1+c-\frac{a+e}{2})\Gamma(1+c-\frac{1-a+e}{2})}.
\end{equation}

In the next theorem we evaluate four specific series.

\begin{theorem} We have that
\begin{align}
&\sum_{k=1}^{\infty}\binom{2k}{k}^3\frac{H_k}{64^k}=\frac{2\pi\left(\pi-3\ln 2\right)}{3\Gamma^4(\frac{3}{4})},&
&\sum_{k=1}^{\infty}\binom{2k}{k}^3\frac{O_k}{64^k}=\frac{ \pi^2}{6\Gamma^4(\frac{3}{4})},
\label{S321}\\
&\sum_{k=1}^{\infty}\binom{2k}{k}^3\frac{H_k^{(2)}}{64^k}
=\frac{ \pi(12G-\pi^2)}{3\Gamma^4(\frac{3}{4})},&
&\sum_{k=1}^{\infty}\binom{2k}{k}^3\frac{O_k^{(2)}}{64^k}
=\frac{ \pi(\pi^2-8G)}{8\Gamma^4(\frac{3}{4})}.
\label{S322}
\end{align}
where $G=\sum_{k=0}^{\infty}\frac{(-1)^k}{(2k+1)^2}$ is the Catalan's constant. 
\end{theorem}
\begin{proof} Let
$$H_k^{(r)}(x)=\sum_{j=0}^{k-1}\frac{1}{(x+j)^r}.$$
Then
$$\frac{d}{dx}\left((x)_k\right)=(x)_k\cdot H_k(x)\quad\mbox{and}\quad
\frac{d}{dx}\left(H_k^{(r)}(x)\right)=-rH_k^{(r+1)}(x).$$
For \eqref{S321}, let $a=b=1/2$ in \eqref{Dixon}, then
$$
\left.\frac{\partial}{\partial c}\left(\pFq{3}{2}{\frac{1}{2},\frac{1}{2},c}{1,\frac{3}{2}-c}{1}\right)\right|_{c=\frac{1}{2}}
=\sum_{k=1}^{\infty}\binom{2k}{k}^3\frac{2O_k+H_k}{64^k}.$$
By setting $b=c=1/2$ in \eqref{Dixon}, we get 
$$
\left.\frac{\partial}{\partial a}\left(
\pFq{3}{2}{a,\frac{1}{2},\frac{1}{2}}{\frac{1}{2}+a,\frac{1}{2}+a}{1}\right)\right|_{a=\frac{1}{2}}
=\sum_{k=1}^{\infty}\binom{2k}{k}^3\frac{2O_k-2H_k}{64^k}.$$
On the other hand, by differentiating the right-hand side of 
\eqref{Dixon} and  \eqref{Whipple} and by using
$$\frac{d}{dx}\left(\Gamma(x)\right)=\Gamma(x)\cdot \Psi(x)\quad\mbox{and}\quad
\frac{d}{dx}\left(\Psi^{(r)}(x)\right)=\Psi^{(r+1)}(x)$$
where $\Psi^{(r)}$ is the polygamma function of order $r$ (with $\Psi^{(0)}=\Psi$), we obtain
$$\left.\frac{\partial}{\partial c}\left(\pFq{3}{2}{\frac{1}{2},\frac{1}{2},c}{1,\frac{3}{2}-c}{1}\right)\right|_{c=\frac{1}{2}}
\!\!\!\!\!\!\!\!
=\pFq{3}{2}{\frac{1}{2},\frac{1}{2},\frac{1}{2}}{1,1}{1}
\cdot\left(-\Psi(1)-\Psi\left(\frac{1}{4}\right)
+\Psi\left(\frac{3}{4}\right)+\Psi\left(\frac{1}{2}\right)\right),
$$
and
$$\left.\frac{\partial}{\partial a}\left(
\pFq{3}{2}{a,\frac{1}{2},\frac{1}{2}}{\frac{1}{2}+a,\frac{1}{2}+a}{1}\right)\right|_{a=\frac{1}{2}}
\!\!\!\!\!\!\!\!
=\pFq{3}{2}{\frac{1}{2},\frac{1}{2},\frac{1}{2}}{1,1}{1}
\cdot\left(
2\Psi(1)-2\Psi\left(\frac{3}{4}\right)-2\ln(2)\right).
$$
By considering a suitable linear combination  of the previous two identities, the special values
$$\Psi\left(\frac{1}{2}\right)-\Psi(1)=-\ln 4,\;
\Psi\left(\frac{1}{4}\right)-\Psi(1)=-\ln 8-\frac{\pi}{2}
,\;
\Psi\left(\frac{3}{4}\right)-\Psi(1)=-\ln 8+\frac{\pi}{2}.
$$
yield immediately \eqref{S321}.

\noindent Let $a=c=1/2$ in \eqref{Whipple},  then
$$
\left.\frac{\partial^2}{\partial e^2}
\left(\pFq{3}{2}{\frac{1}{2},\frac{1}{2},\frac{1}{2}}{e,2-e}{1}\right)\right|_{e=1}
=\sum_{k=1}^{\infty}\binom{2k}{k}^3\frac{2H^{(2)}_k}{64^k}.$$
Moreover, for $c=1/2$, $e=1$ in \eqref{Whipple}, we find
$$
\left.\frac{\partial^2}{\partial a^2}
\left(\pFq{3}{2}{a,1-a,\frac{1}{2}}{1,1}{1}\right)\right|_{a=\frac{1}{2}}
=\sum_{k=1}^{\infty}\binom{2k}{k}^3\frac{-8O^{(2)}_k}{64^k}.$$
On the right-hand side, we have
$$\left.\frac{\partial^2}{\partial e^2}
\left(\pFq{3}{2}{\frac{1}{2},\frac{1}{2},\frac{1}{2}}{e,2-e}{1}\right)\right|_{e=1}
\!\!\!\!\!\!\!\!
=\pFq{3}{2}{\frac{1}{2},\frac{1}{2},\frac{1}{2}}{1,1}{1}
\cdot\left(\frac{\pi^2}{3}-\Psi_1\left(\frac{3}{4}\right)\right),$$
and
$$\left.\frac{\partial^2}{\partial a^2}
\left(\pFq{3}{2}{a,1-a,\frac{1}{2}}{1,1}{1}\right)\right|_{a=\frac{1}{2}}
\!\!\!\!\!\!\!\!
=\pFq{3}{2}{\frac{1}{2},\frac{1}{2},\frac{1}{2}}{1,1}{1}
\cdot\left(\frac{1}{2}\Psi_1\left(\frac{1}{4}\right)
-\frac{1}{2}\Psi_1\left(\frac{3}{4}\right)-\pi^2\right).$$
As before, by combining the results and by using the special values
$$\Psi_1\left(\frac{1}{2}\pm\frac{1}{4}\right)=\pi^2\mp 8G$$
the  conclusion \eqref{S322} easily follows.
\end{proof}

\section{Congruences for the truncated series - Preliminary results}

If $n$ is an odd integer, by replacing $k$ with $(n-k)$ is easy to see that
\begin{equation}\label{CD4}
\sum_{k=0}^n(-1)^k\binom{n}{k}^3=0
\quad\mbox{and}\quad
\sum_{k=0}^n(-1)^k\binom{n}{k}^3H_kH_{n-k}=0.
\end{equation}
The next lemma follows from \cite[Theorem 1]{CC10}.
\begin{lemma}
For any non-negative odd integer $n=2m+1$, we have
\begin{align}\label{CD1}
&\sum_{k=0}^n(-1)^k\binom{n}{k}^3H_k=
-\frac{c_m}{6},\\
\label{CD2}
&\sum_{k=0}^n(-1)^k\binom{n}{k}^3(3H_k^2+H_{k}^{(2)})=
\frac{c_m}{2}
\left(H_m-4H_{2m+1}-H_{3m+2}+2H_{6m+4}\right).
\end{align}
where $\displaystyle c_m=\frac{(-1)^m(6m+3)!(m!)^3}{(3m+1)!((2m+1)!)^3}$.
\end{lemma}

The next lemma establishes some identities involving the harmonic numbers that we will need later on.

\begin{lemma}
For any non-negative integer $n$, we have
\begin{align}\label{CD7}
\sum_{k=0}^n\binom{n}{k}\binom{n+k}{k}\binom{2k}{k}\frac{H_k^{(2)}}{(-4)^k}
=\begin{cases}
\displaystyle \binom{n}{\frac{n}{2}}^2\cdot\frac{\sum_{k=1}^n \frac{(-1)^k}{k^2}}{4^n}
&\text{if $n\equiv_2 0$},\vspace{3mm}\\
\displaystyle \binom{n-1}{\frac{n-1}{2}}^{-2}\cdot \frac{-4^{n-1}}{n^2} 
&\text{if $n\equiv_2 1$}.
\end{cases}
\end{align}
Moreover, for any even integer,
\begin{align}\label{CD5}
&\sum_{k=0}^n\binom{n}{k}\binom{n+k}{k}\binom{2k}{k}\frac{H_k}{(-4)^k}
=\binom{n}{\frac{n}{2}}^2\frac{H_n}{4^n},\\ \label{CD6}
&\sum_{k=0}^n\binom{n}{k}\binom{n+k}{k}\binom{2k}{k}\frac{H_{2k}}{(-4)^k}
=\binom{n}{\frac{n}{2}}^2\frac{H_n}{2\cdot 4^n}.
\end{align}
\end{lemma}
\begin{proof} For $n=2m$, let
$$F(m,k)=\binom{2m}{k}\binom{2m+k}{k}\binom{2k}{k}\binom{2m}{m}^{-2}(-4)^{2m-k}$$
then by Wilf-Zeilberger method we find
$$G(m,k)=-\frac{2(4m+3)k^2}{(2m+1)^3}
\binom{2m+1}{k-1}\binom{2m+k}{k}\binom{2k}{k}\binom{2m}{m}^{-2}(-4)^{2m-k}$$
such that  
$$F(m+1,k)-F(m,k)=G(m,k+1)-G(m,k).$$
Let $S(m)=\sum_{k\geq 1}F(m,k)H_k^{(2)}$ then, by summation by parts (see \cite{JD15} for a similar approach), we have
\begin{align*}
S(m+1)-S(m)&=\sum_{k\geq 0}(G(m,k+1)-G(m,k))H_k^{(2)}\\
&=-\sum_{k\geq 0}\frac{G(m,k+1)}{(k+1)^2}=-\sum_{k\geq 1}\frac{G(m,k)}{k^2}\\
&=-\frac{1}{(2m+1)^2} +\frac{1}{(2m+2)^2}.
\end{align*}
The other identities can be obtained in a similar way.
\end{proof}

The Morita's $p$-adic Gamma function $\Gamma_p$ is defined as the continuous extension to the set of all $p$-adic integers $\mathbb{Z}_p$ of the sequence
$$n\to (-1)^n\prod_{\substack{0\le k < n\\
(k,p)=1}}k$$
where $p$ is an odd prime and $n>1$ is an integer (see \cite[Chapter 7]{Ro00} for a detailed introduction to $\Gamma_p$).
If $x\in \mathbb{Z}_p$  then
$\Gamma_p(0)=1$ and
$$
\Gamma_p(x+1)=
\begin{cases}
-x\Gamma_p(x)\quad&\text{if $|x|_p=1$,}\\
-\Gamma_p(x)\quad &\text{if $|x|_p<1$, }
\end{cases}\
$$
where
$|\cdot|_p$ denotes the $p$-adic norm. By \cite[Theorem 14]{LR16}, for all $a,b\in \mathbb{Z}_p$,
\begin{equation}\label{gmod}
\Gamma_p(a+bp)\equiv_{p^2} \Gamma_p(a)(1+G_1(a) bp)
\end{equation}
where $G_1(a)=\Gamma'_p(a)/\Gamma_p(a)\in \mathbb{Z}_p$.
Moreover 
\begin{equation}\label{pref}
\Gamma_p(x)\Gamma_p(1-x)=(-1)^{s_p(x)}
\end{equation}
where $s_p(x)$ is the integer in $\{1,2,\dots,p\}$ such that $s_p(x)\equiv_p x$. The above formula is the $p$-adic analog of the classic reflection formula for the classic Gamma function
$$\Gamma(x)\Gamma(1-x)=\frac{\pi}{\sin(\pi x)}.$$

\begin{lemma}
For any prime $p>3$, 
\begin{align}\label{CD9}
\frac{\binom{2m}{m}^2}{16^m}
\equiv_{p^2}\begin{cases}
-\Gamma^4_p\left(\frac{1}{4}\right) 
&\text{if $p\equiv_4 1$},\vspace{3mm}\\
16\Gamma^{-4}_p\left(\frac{1}{4}\right) (1+2p)
&\text{if $p\equiv_4 3$},
\end{cases}
\end{align}
where $m=\lfloor p/4\rfloor$. Moreover if $p\equiv_4 3$ then
\begin{equation}\label{CD10}
c_m\equiv_{p^2}\frac{p}{2}\,\Gamma^4_p\left(\frac{1}{4}\right).
\end{equation}
\end{lemma} 
\begin{proof} \noindent We start with \eqref{CD10}. Since $p\equiv_4 3$, we have that $m=(p-3)/4$ and
\begin{align*}
c_m&=\frac{(-1)^m(2m+p)!(m!)^3}{(3m+1)!((2m+1)!)^3}
\equiv_{p^2}\frac{(-1)^{m+1}p(m!)^3}{2(3m+2)!((2m+1)!)^2}\\
&=\frac{(-1)^{m+1}p\,
\Gamma_p^3\left(\frac{p+1}{4}\right)}{2\Gamma_p\left(\frac{3p+3}{4}\right)\Gamma_p^2\left(\frac{p+1}{2}\right)}
\equiv_{p^2}\frac{(-1)^{m+1}p\,
\Gamma_p^3\left(\frac{1}{4}\right)}{2\Gamma_p\left(\frac{3}{4}\right)\Gamma_p^2\left(\frac{1}{2}\right)}
\equiv_{p^2}\frac{p}{2}\,\Gamma^4_p\left(\frac{1}{4}\right)
\end{align*}
where, by \eqref{pref},
\begin{equation}\label{E05}
\Gamma_p^2\left(\frac{1}{2}\right)=(-1)^{\frac{p+1}{2}}=1,\quad
\mbox{and}\quad \Gamma_p\left(\frac{1}{4}\right)
\Gamma_p\left(\frac{3}{4}\right)=(-1)^{\frac{p+1}{4}}=(-1)^{m+1}.
\end{equation}
As regards  \eqref{CD9}, we consider only the case $p\equiv_4 3$ since the other case can be handled similarly. Then  
$$\frac{\binom{2m}{m}^2}{16^m}=
\left(\frac{\left(\frac{1}{2}\right)_m}{\left(1\right)_m}\right)^2
=\frac{\Gamma_p^2\left(1\right)}{\Gamma_p^2\left(\frac{1}{2}\right)}\cdot\frac{\Gamma_p^2\left(\frac{1}{2}+m\right)}{\Gamma_p^2\left(1+m\right)}
=
\frac{\Gamma_p^2\left(-\frac{1}{4}+\frac{p}{4}\right)}{
\Gamma_p^2\left(\frac{1}{4}+\frac{p}{4}\right)}.$$
By \eqref{gmod} and by \cite[Lemma 2.4]{Li16},
$$
\Gamma_p\left(-\frac{1}{4}+\frac{p}{4}\right)\equiv_{p^2}
\Gamma_p\left(-\frac{1}{4}\right)\left(1+(G_1(1)+H_{3m+1})\frac{p}{4}\right),
$$
and 
$$\Gamma_p\left(\frac{1}{4}+\frac{p}{4}\right)\equiv_{p^2}
\Gamma_p\left(\frac{1}{4}\right)\left(1+(G_1(1)+H_m)\frac{p}{4}\right).$$
Therefore, since $\Gamma_p\left(-\frac{1}{4}\right)=4\Gamma_p\left(-\frac{3}{4}\right)$,
\begin{align*}
\frac{\binom{2m}{m}^2}{16^m}&\equiv_{p^2}
\frac{\Gamma_p^2\left(-\frac{1}{4}\right)}{\Gamma_p^2\left(\frac{1}{4}\right)}\cdot\left(1+(H_{3m+1}-H_{m})\frac{p}{2}\right)
\equiv_{p^2} 16\Gamma^{-4}_p\left(\frac{1}{4}\right) (1+2p)
\end{align*}
where we also used  \eqref{E05} and
$$H_{3m+1}=H_{p-1}-\sum_{j=1}^{m+1}\frac{1}{p-j}\equiv_p H_m+\frac{4}{3p+1}\equiv_p H_m+4.$$
\end{proof}

\section{Congruences for the truncated series - Main results}
 
\begin{theorem} For any prime $p>3$,
\begin{align}\label{C321}
&\sum_{k=1}^{p-1}\binom{2k}{k}^3\frac{H_k}{64^k}
\equiv_{p^2}
\begin{cases}
\Gamma^4_p\left(\frac{1}{4}\right) \cdot (2q_p(2)-pq^2_p(2))
&\text{if $p\equiv_4 1$},\vspace{3mm}\\
-\frac{p}{12}\Gamma^4_p\left(\frac{1}{4}\right) 
&\text{if $p\equiv_4 3$},
\end{cases}
\end{align}
and 
\begin{align}\label{C322}
&\sum_{k=1}^{p-1}\binom{2k}{k}^3\frac{H_k^{(2)}}{64^k}
\equiv_{p^2}\begin{cases}
-\Gamma_p^4\left(\frac{1}{4}\right) \cdot (4E_{p-3}-2E_{2p-4})
&\text{if $p\equiv_4 1$},\vspace{3mm}\\
-\frac{1}{4}\Gamma_p^4\left(\frac{1}{4}\right)
&\text{if $p\equiv_4 3$}.
\end{cases}
\end{align}
\end{theorem}
\begin{proof} For \eqref{C322}, if $p\equiv_4 1$ then  $n=(p-1)/2$ is even and we use  \eqref{E01} and \eqref{CD7}. Finally we use \eqref{CD9}. 
If $p\equiv_4 3$ then $n=(p-1)/2=2m+1$ is odd.
We have
$$\frac{\binom{2k}{k}}{(-4)^k\binom{n}{k}}=
\prod_{j=0}^{k-1}\left(1-\frac{p}{2j+1}\right)^{-1}
\equiv_{p^2}1-\frac{p}{2}\sum_{j=0}^{k-1}\frac{1}{n-j}
=1-\frac{p}{2}\left(H_n-H_{n-k}\right)$$
and therefore
\begin{equation}\label{E02}
\frac{1}{4^k}\binom{2k}{k}\equiv_{p^2}
(-1)^k\binom{n}{k}\left(1-\frac{p}{2}\left(H_n-H_{n-k}\right)\right).
\end{equation}
Thus, by \eqref{CD4}, \eqref{CD5}, and \eqref{CD10}, 
\begin{align*}
\sum_{k=1}^{p-1}\binom{2k}{k}^3\frac{H_k}{64^k}
&\equiv_{p^2}
\sum_{k=0}^{n}
(-1)^k\binom{n}{k}^3\left(1-\frac{3p}{2}\left(H_n-H_{n-k}\right)\right)H_k\\
&\equiv_{p^2}
\left(1-\frac{3p}{2}H_n\right)\left(-\frac{c_m}{6}\right)
\equiv_{p^2}
-\frac{p}{12}\Gamma^4_p\left(\frac{1}{4}\right).
\end{align*}
As regards \eqref{C322}, we use \eqref{E01} and \eqref{CD7} with $n=(p-1)/2$. Then we apply \eqref{CD9} and \eqref{E03}.
\end{proof}

\begin{theorem} For any prime $p>3$,
\begin{align}\label{E09}
&\sum_{k=1}^{p-1}\binom{2k}{k}^3\frac{O_k}{64^k}
\equiv_{p^2}
\begin{cases}
0
&\text{if $p\equiv_4 1$},\vspace{3mm}\\
-\frac{p}{12}\Gamma^4_p\left(\frac{1}{4}\right) 
&\text{if $p\equiv_4 3$},
\end{cases}
\end{align}
and 
\begin{align}\label{E10}
&\sum_{k=1}^{p-1}\binom{2k}{k}^3\frac{O^{(2)}_k}{64^k}
\equiv_{p}
\begin{cases}
\frac{1}{2}\Gamma^4_p\left(\frac{1}{4}\right) E_{p-3}
&\text{if $p\equiv_4 1$},\vspace{3mm}\\
-\frac{1}{16}\Gamma^4_p\left(\frac{1}{4}\right) 
&\text{if $p\equiv_4 3$}.
\end{cases}
\end{align}
\end{theorem}
\begin{proof}
If $n=(p-1)/2$ is even then by \eqref{E01}, \eqref{CD5} and  \eqref{CD6}, 
$$\sum_{k=1}^{p-1}\binom{2k}{k}^3\frac{O_k}{64^k}
=\sum_{k=1}^{p-1}\frac{\binom{2k}{k}^3}{64^k}\left(H_{2k}-\frac{H_k}{2}\right)
\equiv_{p^2}\binom{n}{\frac{n}{2}}^2\frac{H_n}{2\cdot 4^n}
-\frac{1}{2}\binom{n}{\frac{n}{2}}^2\frac{H_n}{4^n}\equiv_{p^2} 0.$$
Assume now that $n=(p-1)/2=2m+1$ is odd. We have that 
$$H_{2(n-k)}=H_{p-1}-\sum_{j=1}^{2k}\frac{1}{p-j}\equiv_{p^2}H_{2k}+pH^{(2)}_{2k}\equiv_{p^2}
H_{2k}+\frac{p}{4}\left(H_k^{(2)}-H_{n-k}^{(2)}\right).$$
Hence
\begin{align*}
\sum_{k=1}^{n}
(-1)^k\binom{n}{k}^3H_{2k}&=-\sum_{k=0}^{n}
(-1)^k\binom{n}{n-k}^3H_{2(n-k)}\\
&\equiv_{p^2}
-\sum_{k=1}^{n}
(-1)^k\binom{n}{k}^3H_{2k}-\frac{p}{4}\sum_{k=0}^{n}
(-1)^k\binom{n}{k}^3\left(H_k^{(2)}-H_{n-k}^{(2)}\right)
\end{align*}
which implies 
\begin{equation}\label{E07}
\sum_{k=1}^{n}
(-1)^k\binom{n}{k}^3H_{2k}
\equiv_{p^2}
-\frac{p}{8}\sum_{k=0}^{n}
(-1)^k\binom{n}{k}^3\left(H_k^{(2)}-H_{n-k}^{(2)}\right)
=
-\frac{p}{4}\sum_{k=0}^{n}
(-1)^k\binom{n}{k}^3H_k^{(2)}.
\end{equation}
Moreover
\begin{equation}\label{E06}
H_{2k}=\frac{H_k}{2}+\sum_{j=0}^{k-1}\frac{1}{2j+1}
\equiv_{p}
\frac{1}{2}\left(H_k-\sum_{j=0}^{k-1}\frac{1}{n-j}\right)
\equiv_{p}
\frac{1}{2}\left(H_k+H_{n-k}-H_n\right).
\end{equation}
Consequently by \eqref{E02}, \eqref{E07}, \eqref{E06}, 
\begin{align}\label{E25}
\sum_{k=1}^{p-1}\binom{2k}{k}^3\frac{H_{2k}}{64^k}
&\equiv_{p^2}
\sum_{k=1}^{n}
(-1)^k\binom{n}{k}^3\left(1-\frac{3p}{2}\left(H_n-H_{n-k}\right)\right)H_{2k}\nonumber\\
&\equiv_{p^2}
\sum_{k=1}^{n}
(-1)^k\binom{n}{k}^3H_{2k}-\frac{3p}{4}\sum_{k=0}^{n}
(-1)^k\binom{n}{k}^3\left(H_n-H_{n-k}\right)\left(H_k+H_{n-k}-H_n\right)\nonumber\\
&\equiv_{p^2}
-\frac{p}{4}\sum_{k=0}^{n}
(-1)^k\binom{n}{k}^3H_k^{(2)}-\frac{3p}{2}H_n\sum_{k=0}^{n}
(-1)^k\binom{n}{k}^3H_k\nonumber\\
&\quad+\frac{3p}{4}\sum_{k=0}^{n}
(-1)^k\binom{n}{k}^3H^2_{n-k}
-\frac{3p}{4}H_n\sum_{k=0}^{n}
(-1)^k\binom{n}{k}^3H_{n-k}\nonumber\\
&\equiv_{p^2}
-\frac{p}{4}\sum_{k=0}^{n}
(-1)^k\binom{n}{k}^3\left(3H^2_{k}+H_k^{(2)}\right)\equiv_{p^2}
-\frac{c_m}{4}\equiv_{p^2}-\frac{p}{8} \Gamma^4_p\left(\frac{1}{4}\right)
\end{align}
where in the last step we used \eqref{CD1}, \eqref{CD2},
and  \eqref{CD10} (note that $m<2m+1<3m+2<p<6m+4<2p$).
Finally, by \eqref{C321},
$$\sum_{k=1}^{p-1}\binom{2k}{k}^3\frac{O_k}{64^k}
=\sum_{k=1}^{p-1}\frac{\binom{2k}{k}^3}{64^k}\left(H_{2k}-\frac{H_k}{2}\right)
\equiv_{p^2}-\frac{p}{12}\Gamma^4_p\left(\frac{1}{4}\right)
$$
and the proof of \eqref{E09} is complete. 

As regards \eqref{E10}, we have by \eqref{E04}
\begin{equation}\label{E11}
O_k^{(2)}=\sum_{j=0}^{k-1}\frac{1}{(2j+1)^2}\equiv_p
\sum_{j=0}^{k-1}\frac{1}{4(n-j)^2}
=\frac{H_n^{(2)}-H_{n-k}^{(2)}}{4}\equiv_p-\frac{H_{n-k}^{(2)}}{4}
\end{equation}
where $n=(p-1)/2$.
Then, by \eqref{E02} and \eqref{E11},
\begin{align*}
\sum_{k=1}^{p-1}\binom{2k}{k}^3\frac{O_k^{(2)}}{64^k}
&\equiv_{p}-\frac{1}{4}\sum_{k=0}^{n}
(-1)^k\binom{n}{k}^3H_{n-k}^{(2)}
=\frac{(-1)^{n+1}}{4}\sum_{k=0}^{n}
(-1)^k\binom{n}{k}^3H_{k}^{(2)}\\
&\equiv_{p}\frac{(-1)^{n+1}}{4}\sum_{k=1}^{p-1}
\binom{2k}{k}^3\frac{H_{k}^{(2)}}{64^k}
\end{align*}
and the desired result follows from \eqref{C322}.
\end{proof}

\begin{remark}  By \eqref{C321}, \eqref{E09}, and \eqref{E21}
for any  prime $p>3$ 
$$\sum_{k=1}^{p-1}\binom{2k}{k}^3\frac{H_{2k}-H_{k}}{64^k}\equiv_p
q_p(2)\sum_{k=0}^{p-1}\frac{\binom{2k}{k}^3}{64^k}$$
which is a particular case of \cite[Corollary 2]{KLMSY16}.
Moreover by \eqref{C321} and \eqref{E25},
if $p\equiv_4 3$ then
$$\sum_{k=1}^{p-1}\binom{2k}{k}^3\frac{H_{k}}{64^k}\equiv_p
\sum_{k=1}^{p-1}\binom{2k}{k}^3\frac{H_{2k}}{64^k}\equiv_p 0$$
which appears in \cite{Szw12}.
\end{remark}

\section{Coda}

In this final section we present a few more results with the same flavor related to ${}_2F_1(1/2)$, ${}_4F_3(-1)$ and ${}_6F_5(-1)$ . 

\noindent By Bailey's theorem \cite[p.11]{Ba35},
\begin{equation}\label{Bailey}
\pFq{2}{1}{a,1-a}{c}{\frac{1}{2}}=\frac{\Gamma(\frac{c}{2})\Gamma(\frac{c+1}{2})}{\Gamma(\frac{a+c}{2})\Gamma(\frac{ 1-a+c}{2})}
\end{equation}
it follows that
\begin{equation}\label{E22}
\sum_{k=0}^{\infty}\frac{\binom{2k}{k}^2}{32^k}=
\pFq{2}{1}{\frac{1}{2},\frac{1}{2}}{1}{\frac{1}{2}}=
\frac{ \sqrt{\pi}}{\Gamma^2(\frac{3}{4})}.
\end{equation}
Moreover, it has been proved (see for example \cite[Corollary 2.2]{Szh14} and \cite[(1.4)]{Li16})
\begin{equation}\label{E23}
\sum_{k=0}^{p-1}\frac{\binom{2k}{k}^2}{32^k}
\equiv_{p^2} \left \{ \begin{array}{ll}  (-1)^{\frac{p+1}{2}}\Gamma_p\left(\frac{1}{2}\right)\Gamma_p\left(\frac{1}{4}\right)^2 & \text{ if } p\equiv_4 1,\\
 0  & \text{ if } p\equiv_4 3. \end{array} \right. 
\end{equation}
Note that, by Clausen's Formula and its truncated version \cite[Lemma18]{CDLNS13}, equations \eqref{E20}, \eqref{E22}, and congruences \eqref{E21}, \eqref{E23} satisfy the following connecting relationships,
$$\left(\sum_{k=0}^{\infty}\frac{\binom{2k}{k}^2}{32^k}\right)^2
=\sum_{k=0}^{\infty}\frac{\binom{2k}{k}^3}{64^k}\and
\left(\sum_{k=0}^{p-1}\frac{\binom{2k}{k}^2}{32^k}\right)^2
\equiv_{p^2}\sum_{k=0}^{p-1}\frac{\binom{2k}{k}^3}{64^k}.$$
Now by letting $a=1/2$ in \eqref{Bailey},  then 
$$\sum_{k=1}^{\infty}\binom{2k}{k}^2
\frac{H_k}{32^k}=\left.\frac{\partial}{\partial c}\left(
\pFq{2}{1}{\frac{1}{2},\frac{1}{2}}{c}{\frac{1}{2}}
\right)\right|_{c=1}=\left.\frac{\partial}{\partial c}\left(
\frac{2^{1-c}\sqrt{\pi}\,\Gamma(c)}{
\Gamma^2(\frac{1}{4}+\frac{c}{2})}
\right)\right|_{c=1}
=\frac{\sqrt{\pi}\left(\pi-4\ln 2\right)}{2\Gamma^2(\frac{3}{4})}
.$$
The following result yields a $p$-adic analog of the above series.

\begin{theorem} For any prime $p>3$,
\begin{align}\label{E70}
&\sum_{k=0}^{p-1}\binom{2k}{k}^2\frac{H_k}{32^k}
\equiv_{p^2}
\begin{cases}
\Gamma_p\left(\frac{1}{2}\right)\Gamma^2_p\left(\frac{1}{4}\right) \cdot (2q_p(2)-pq^2_p(2))
&\text{if $p\equiv_4 1$},\vspace{3mm}\\
\frac{1}{2}
\Gamma_p\left(\frac{1}{2}\right)\Gamma^2_p\left(\frac{1}{4}\right) 
&\text{if $p\equiv_4 3$}.
\end{cases}
\end{align}
\end{theorem}
\begin{proof}
It suffices to use the identity, 
\begin{align*}
\sum_{k=0}^n\binom{n}{k}\binom{n+k}{k}\frac{H_k}{(-2)^k}
=\begin{cases}
\displaystyle \binom{n}{\frac{n}{2}}^2\cdot\frac{(-1)^{\frac{n}{2}}H_n}{2^n}
&\text{if $n\equiv_2 0$},\vspace{3mm}\\
\displaystyle \binom{n-1}{\frac{n-1}{2}}^{-1}\cdot 
\frac{(-1)^{\frac{n+1}{2}}\,2^{n-1}}{n} 
&\text{if $n\equiv_2 1$},
\end{cases}
\end{align*}
then the verification of congruence \eqref{E70} can be carried out along the lines of the proof \eqref{C321}. The interested reader may fill in the necessary details. 
\end{proof}

By a couple of formulas which appear in \cite[(2) and (3) p.28]{Ba35},
\begin{align}
&\label{whipple43}
\pFq{4}{3}{a,1+\frac{a}{2}, d, e}{\frac{a}{2},1+a-d,1+a-e}{-1}=
\frac{\Gamma(1+a-d)\Gamma(1+a-e)}{\Gamma(1+a)\Gamma(1+a-d-e)},\\
&\label{whipple65}
\pFq{6}{5}{a,1+\frac{a}{2}, b, c, d,e}{\frac{a}{2},1+a-b,1+a-c,1+a-d,1+a-e}{-1}\\
&\qquad\qquad\qquad=
\frac{\Gamma(1+a-d)\Gamma(1+a-e)}{\Gamma(1+a)\Gamma(1+a-d-e)}
\,\pFq{3}{2}{1+a-b-c,d, e}{1+a-b,1+a-c}{1},
\end{align}
we have the series labeled (B.1) and (A.1) in \cite{VH97}
\begin{align*}
&\sum_{k=0}^{\infty}(4k+1)\frac{\binom{2k}{k}^3}{64^k}(-1)^k=
\pFq{4}{3}{\frac{1}{2},\frac{5}{4},\frac{1}{2},\frac{1}{2}}{\frac{1}{4},1,1}{-1}=\frac{2}{\pi},\\
&\sum_{k=0}^{\infty}(4k+1)\frac{\binom{2k}{k}^5}{1024^k}(-1)^k
=\pFq{6}{5}{\frac{1}{2},\frac{5}{4},\frac{1}{2},\frac{1}{2},\frac{1}{2},\frac{1}{2}}{\frac{1}{4},1,1,1,1}{-1}=\frac{2}{\Gamma^4\left(\frac{3}{4}\right)}.
\end{align*}
Moreover, it has been shown that the $p$-analogs (B.2) and (A.2) in \cite{VH97} hold for any prime $p>3$ (see  \cite{MO08} and \cite{Lo11}),
\begin{align*}
&\sum_{k=0}^{p-1}(4k+1)\frac{\binom{2k}{k}^3}{64^k}(-1)^k
\equiv_{p^3} (-1)^{\frac{p-1}{2}} p\\
&\sum_{k=0}^{p-1}(4k+1)\frac{\binom{2k}{k}^5}{1024^k}(-1)^k
\equiv_{p^3}
\begin{cases}
\displaystyle -\frac{p}{\Gamma^4_p\left(\frac{3}{4}\right)} 
&\text{if $p\equiv_4 1$},\vspace{3mm}\\
0
&\text{if $p\equiv_4 3$},
\end{cases}
\end{align*}
Recently Guillera proved  this elegant Ramanujan-type formula involving harmonic numbers \cite[(32)]{Gu13},
\begin{equation}\label{E61}
\sum_{k=0}^{\infty}\frac{\binom{2k}{k}^3}{64^k}\left(2-3(4k+1)H_k\right)(-1)^k
=\frac{12\ln 2}{\pi}
\end{equation}
The above evaluation can be established by noting that the left-hand side is equal to
\begin{align*}
2\pFq{4}{3}{\frac{1}{2},\frac{5}{4},\frac{1}{2},\frac{1}{2}}{\frac{1}{4},1,1}{-1}&+
\left.\frac{\partial}{\partial a}\left(
\pFq{4}{3}{a,1+\frac{a}{2}, \frac{1}{2}, \frac{1}{2}}{\frac{a}{2},\frac{1}{2}+a,\frac{1}{2}+a}{-1}\right)\right|_{a=\frac{1}{2}}
\!\!\!\!\!\!
-\left.\frac{\partial}{\partial e}\left(
\pFq{4}{3}{\frac{1}{2},\frac{5}{4}, \frac{1}{2}, e}{\frac{1}{4},1,\frac{3}{2}-e}{-1}\right)\right|_{e=\frac{1}{2}}.
\end{align*}
Then by using \eqref{whipple43} we obtain the right-hand side. In a similar way, from \eqref{whipple65}, we can get
\begin{equation}\label{E62}
\sum_{k=0}^{\infty}\frac{\binom{2k}{k}^5}{1024^k}\left(2-5(4k+1)H_k\right)(-1)^k
=\frac{4(15\ln(2)-2\pi)}{3\Gamma^4\left(\frac{3}{4}\right)}.
\end{equation}
The infinite series \eqref{E61} and \eqref{E62} have $p$-adic analogs which are given in the next result.

\begin{theorem} For any prime $p>3$,
\begin{align}
&\label{E63}
\sum_{k=0}^{p-1}\frac{\binom{2k}{k}^3}{64^k}\left(2-3(4k+1)H_k\right)(-1)^k
\equiv_{p^2}(-1)^{\frac{p-1}{2}}(2+6pq_p(2)),\\
&\label{E64}
\sum_{k=0}^{p-1}\frac{\binom{2k}{k}^4}{256^k}\left(2-4(4k+1)H_k\right)
\equiv_{p^2}2+12pq_p(2),\\ 
&\label{E65}
\sum_{k=0}^{p-1}\frac{\binom{2k}{k}^5}{1024^k}\left(2-5(4k+1)H_k\right)(-1)^k
\equiv_{p^2}
\left \{ \begin{array}{ll} - (2+10pq_p(2))\Gamma^4_p\left(\frac{1}{4}\right) & \text{ if } p\equiv_4 1, \vspace{3mm}\\
 0
& \text{ if } p\equiv_4 3. \end{array} \right.
%\sum_{k=0}^{p-1}\frac{\binom{2k}{k}^3}{64^k}. 
\end{align}
\end{theorem}
\begin{proof} Let $n=\frac{p-1}{2}$, then, by \eqref{E02}, the left-hand side of \eqref{E63} is congruent modulo $p^2$ to
\begin{align*}
\sum_{k=0}^{n}&
\binom{n}{k}^3\left(1-\frac{3p}{2}\left(H_n-H_{n-k}\right)\right)
\left(2+3(2n-4k-p)H_k\right)\\
&\equiv_{p^2}(2-3pH_n)\sum_{k=0}^{n}
\binom{n}{k}^3
\left(1+3(n-2k)H_k\right)
+3p\sum_{k=0}^{n}
\binom{n}{k}^3H_{n-k}\\
&\qquad +9p\sum_{k=0}^{n}
\binom{n}{k}^3
(n-2k)H_kH_{n-k}
-3p\sum_{k=0}^{n}
\binom{n}{k}^3H_k\\
&\equiv_{p^2}(-1)^n(2-3p H_n)\equiv_{p^2} (-1)^{\frac{p-1}{2}}(2+6pq_p(2))
\end{align*}
where  we used \eqref{E13} and the identities 
$$\sum_{k=0}^{n}\binom{n}{k}^3 \left(1+3(n-2k)\right)H_k=(-1)^n\and
\sum_{k=0}^{n}\binom{n}{k}^3(n-2k)H_kH_{n-k}=0.$$
The first one follows from \cite[(3)]{PS03}, whereas the second one follows by replacing $k$ with $(n-k)$.

Congruences \eqref{E64} and \eqref{E65} can be obtained in a similar way by using the identities 
\begin{align*}
&\sum_{k=0}^{n}\binom{n}{k}^4 \left(1+4(n-2k)\right)H_k=(-1)^n\binom{2n}{n},\\
&\sum_{k=0}^{n}\binom{n}{k}^5 \left(1+5(n-2k)\right)H_k=
(-1)^n\sum_{k=0}^{n}\binom{n}{k}^2\binom{n+k}{k},
\end{align*}
which are equivalent to \cite[(4)]{PS03} and \cite[(5)]{PS03} respectively.
\end{proof}

\medskip

\end{document}